\documentclass[11pt]{article}
\usepackage[english]{babel}
\usepackage[a4paper,scale=0.75,hcentering]{geometry}
\usepackage{amsmath,amsthm,amssymb,mathtools,enumitem}
\usepackage{xcolor}
\usepackage{setspace}
\allowdisplaybreaks
\usepackage[hyperfootnotes=false]{hyperref}
\hypersetup{colorlinks=true,linkcolor=blue,urlcolor=red,citecolor=blue,
  pdftitle={Sharp quantitative stability of the Hong--Krahn--Szeg\H{o} inequality},
  pdfauthor={Yuxuan Zhou, Wenming Zou},
  pdfkeywords={Dirichlet eigenvalues, Hong--Krahn--Szego inequality, quantitative stability, shape optimization, torsional rigidity}}
\usepackage[nameinlink,capitalise,sort]{cleveref}
\crefname{equation}{}{}
\crefname{enumi}{}{}
\newcounter{results}[section]

\newtheorem{theorem}[results]{Theorem}
\newtheorem{lemma}[results]{Lemma}
\newtheorem{proposition}[results]{Proposition}

\theoremstyle{remark}

\numberwithin{equation}{section}
\newcommand{\R}{\mathbb R}\newcommand{\om}{\omega_d}
\newcommand{\lam}{\lambda_B}\newcommand{\A}{\mathcal A_2}
\newcommand{\HKS}{\mathcal D}\newcommand{\sym}{\mathbin\triangle}
\newcommand{\eps}{\varepsilon}\newcommand{\dd}{\,dx}
\DeclareMathOperator{\diam}{diam}
\title{\bfseries Sharp quantitative stability of the Hong--Krahn--Szeg\H{o} inequality\thanks{This work is supported by the National Key R\&D Program of China (Grant 2023YFA1010001) and NSFC~(12571123). E-mail addresses: \href{mailto:zhou-yx22@mails.tsinghua.edu.cn}{\nolinkurl{zhou-yx22@mails.tsinghua.edu.cn}}~(Zhou), \href{mailto:zou-wm@mail.tsinghua.edu.cn}{\nolinkurl{zou-wm@mail.tsinghua.edu.cn}}~(Zou).}}
\author{{\bfseries Yuxuan Zhou, Wenming Zou}\\
  {\footnotesize\itshape Department of Mathematical Sciences, Tsinghua University, Beijing 100084, China.}}
\date{}
\begin{document}
\maketitle
\begin{abstract}
\small
We establish sharp quantitative stability estimates for the Hong--Krahn--Szeg\H{o} inequality for arbitrary open sets of positive finite measure in $\R^d$, $d\ge2$. In dimensions $d\ge3$, the optimal stability exponent is $(d+2)/2$, while in dimension two the sharp stability modulus is $t^2/\log(e/t)$. These results improve the quantitative stability estimate of Brasco and Pratelli~\cite[Theorem 3.5]{BP} and resolve Open Problem 7.56 of Brasco and De Philippis~\cite[Open Problem 7.56]{BDP}.


\end{abstract}

\section{Introduction}\label{sec:introduction}
Let $d\ge2$, $B=B_1(0)\subset\R^d$, $\om=|B|$, and $\lam=\lambda_1(B)$. For an open set $\Omega\subset\R^d$ with $0<|\Omega|<\infty$, let $\lambda_k(\Omega)$ denote the Dirichlet eigenvalues of $-\Delta$, counted with multiplicity and defined by the quadratic form $\int_\Omega|\nabla u|^2\dd$ on $H^1_0(\Omega)$. The Rayleigh--Faber--Krahn inequality~\cite{Faber23,Krahn25,Krahn26} states that the ball minimizes the first eigenvalue at fixed volume:
\begin{equation}\label{eq:classicalFK}
 |\Omega|^{2/d}\lambda_1(\Omega)\ge\om^{2/d}\lam\,.
\end{equation}
Equality holds precisely for balls, up to sets of Sobolev capacity zero; see~\cite[(1.8), $q=2$]{BDPV}. As for the second eigenvalue, the Hong--Krahn--Szeg\H{o} inequality~\cite{Krahn26,Hong54} gives
\begin{equation}\label{eq:classicalHKS}
 |\Omega|^{2/d}\lambda_2(\Omega)\ge(2\om)^{2/d}\lam\,.
\end{equation}
The minimizing configuration is a disjoint union of two equal balls, again up to sets of Sobolev capacity zero. The formulation for arbitrary finite-measure open sets and its equality case are given in~\cite[Section 7.6.1]{BDP}.

In this paper, we study the quantitative stability of the Hong--Krahn--Szeg\H{o} inequality \eqref{eq:classicalHKS}. We seek estimates that control the distance from a set to the family of unions of two disjoint equal balls in terms of its spectral deficit, with the best possible rate. We do not attempt a detailed review of quantitative spectral inequalities. For background, see Henrot~\cite{Henrot2006} and Brasco and De Philippis~\cite{BDP}; see also the earlier stability results of Melas~\cite{Melas1992} and Fusco, Maggi and Pratelli~\cite{FMP09}, and the recent work on the Dirichlet spectrum by Bucur, Lamboley, Nahon and Prunier~\cite{BLNP2026}.

For the first eigenvalue, the Fraenkel asymmetry is
\begin{equation*}
 \mathcal A(\Omega)=\inf\left\{\frac{|\Omega\sym B'|}{|\Omega|}:B'\text{ is a ball},\ |B'|=|\Omega|\right\}.
\end{equation*}
Brasco--De Philippis--Velichkov~\cite[Main Theorem, $q=2$]{BDPV} proved the sharp quadratic estimate
\begin{equation}\label{eq:sharpFK}
 |\Omega|^{2/d}\lambda_1(\Omega)-\om^{2/d}\lam\ge c_d\mathcal A(\Omega)^2
\end{equation}
for every open set of positive finite measure.

For the HKS problem, define
\begin{align}
 \HKS(\Omega)&=|\Omega|^{2/d}\lambda_2(\Omega)-(2\om)^{2/d}\lam\,,\label{eq:defD}\\
 \A(\Omega)&=\inf\left\{\frac{|\Omega\sym(B_1\cup B_2)|}{|\Omega|}:B_1\cap B_2=\varnothing,\ |B_1|=|B_2|=|\Omega|/2\right\}\,.\label{eq:defA}
\end{align}
Here $B_1,B_2$ are open balls, and tangency is allowed. The quantity $\A$ is the Fraenkel $2$-asymmetry. Brasco--Pratelli~\cite[Theorem 3.5]{BP} proved
\begin{equation}\label{eq:historicalBP}
 \HKS(\Omega)\ge c_d\A(\Omega)^{2(d+1)}
\end{equation}
for bounded open sets. The optimality of this exponent was left open. As explained in~\cite[Remark 3.6]{BP}, if the single-ball stability estimate has exponent $\kappa$, their argument gives the HKS exponent $\kappa(d+1)/2$. Thus based on the estimate \eqref{eq:sharpFK},their argument yields exponent $d+1$. Brasco--De Philippis~\cite[Theorem 7.51]{BDP} record the resulting estimate
\begin{equation}\label{eq:historicalBDP}
 \HKS(\Omega)\ge c_d\A(\Omega)^{d+1}
\end{equation}
for every open set of positive finite measure. In~\cite[Open Problem 7.56]{BDP}, they ask whether $d+1$ is optimal and, if not, what the correct exponent is.

In this paper we resolve this open problem. Set $\Phi_d(0)=0$ and, for $0<a\le2$, define
\begin{equation}\label{eq:Phi}
 \Phi_d(a)=\begin{cases}
 a^2/\log(e/a),&d=2,\\
 a^{(d+2)/2},&d\ge3.
 \end{cases}
\end{equation}
We prove the following stability estimate. 
\begin{samepage}
\begin{theorem}\label{thm:main}
For every $d\ge2$ there exists $c_d>0$ such that every open $\Omega\subset\R^d$ with $0<|\Omega|<\infty$ satisfies
\begin{equation}\label{eq:main}
 \HKS(\Omega)\ge c_d\Phi_d(\A(\Omega))\,.
\end{equation}
\end{theorem}
\end{samepage}

Our next theorem shows that the estimate \eqref{eq:main} is sharp even among smooth sets with two components.

\newpage
\begin{samepage}
\begin{theorem}\label{thm:sharp}
For every $d\ge2$ and for all sufficiently small $\eps>0$, there are smooth bounded open sets $\Omega_\eps$, each the union of two congruent disjoint domains, and two dimensional positive constants $c_d,C_d$ such that
\begin{equation}\label{eq:sharp}
 c_d\eps\le\A(\Omega_\eps)\le C_d\eps,
 \qquad c_d\Phi_d(\eps)\le\HKS(\Omega_\eps)\le C_d\Phi_d(\eps).
\end{equation}
\end{theorem}
\end{samepage}
Let us briefly outline our strategy of the proof.

To prove Theorem~\ref{thm:main}, we first reduce the problem to two disjoint spectral pieces, following~\cite[Lemma 3.1]{BP} and~\cite[Lemma 7.49]{BDP}. By the Kohler--Jobin inequality, a small spectral deficit also gives small torsion deficits for these pieces. Combining the cutoff estimates and the quantitative torsion estimates of Brasco, De Philippis and Velichkov~\cite{BDPV} with the resolvent estimate of Allen, Kriventsov and Neumayer~\cite{AKN}, we obtain Lemma~\ref{lem:prepare}. This allows us to retain two bounded disjoint subsets whose torsion functions satisfy both an interior gradient estimate and a full-space $L^2$ estimate relative to balls centered at their barycenters. The main difficulty is that these comparison balls may overlap, even though the subsets themselves are disjoint. To overcome this difficulty, we exploit the disjointness more directly. On the intersection of the comparison balls, the two torsion functions cannot both be positive, which allows us to construct auxiliary functions with a positive lower bound on this region. In Proposition~\ref{prop:contact}, we combine this observation with elementary averaging estimates and the Poincar\'e inequality to control the overlap sharply. This gives the optimal stability exponent in dimensions $d\ge3$ and the logarithmic correction in dimension two.

For Theorem~\ref{thm:sharp}, we use the nearly tangent two-ball geometry of~\cite[Example 3.4]{BP} and the truncated-ball eigenvalue estimate of van den Berg~\cite[Theorem 1(ii)]{vdB}. Smooth inner approximations give the higher-dimensional examples. In dimension two, a logarithmic displacement is needed to obtain the sharp modulus. Lemma~\ref{lem:rayleigh} controls the spectral deficit by the deformation energy, while Lemma~\ref{lem:matching} ensures that the Fraenkel $2$-asymmetry is of order $\eps$.

In Section~\ref{sec:preparation} we collect  preliminary estimates. Section~\ref{sec:proofs} is then devoted to the proofs of main theorems.

\section{Preliminary results}\label{sec:preparation}
For an open set $E$ of positive finite measure, let $w_E\in H^1_0(E)$ solve $-\Delta w_E=1$ weakly. All torsion functions are extended by zero to $\R^d$. Write
\begin{gather*}
 T(E)=\int_Ew_E\dd=\int_E|\nabla w_E|^2\dd,\qquad \mathcal E(E)=-\tfrac12T(E),\\*
 D(E)=\mathcal E(E)|E|^{-(d+2)/d}-\mathcal E(B)\om^{-(d+2)/d}\,.
\end{gather*}
The torsion deficit $D(E)$ is scale invariant, and the ball torsion is
\begin{equation}\label{eq:balltorsion}
 w_{B_r(c)}(x)=\frac{(r^2-|x-c|^2)_+}{2d}\,.
\end{equation}
For bounded open set $U$ with $|U|=\om$, let $b_U=\om^{-1}\int_Ux\dd$ and define
\begin{equation*}
 \alpha(U)=\int_{U\sym B_1(b_U)}\big|1-|x-b_U|\big|\dd\,.
\end{equation*}

The first lemma presents the classical Saint--Venant inequality and the Kohler--Jobin inequality; see~\cite[(2.5)--(2.6), $q=2$]{BDPV} for example.
\begin{lemma}\label{lem:classical}
Let $d\ge2$, and let $E\subset\R^d$ be open with $0<|E|<\infty$. Then
\begin{align}
 T(E)&\le T(B)\left(\frac{|E|}{\om}\right)^{(d+2)/d},\label{eq:SV}\\*
 \lambda_1(E)T(E)^{2/(d+2)}&\ge\lam T(B)^{2/(d+2)}\,.\label{eq:KJ}
\end{align}
Equality in~\eqref{eq:SV} holds precisely for balls, up to sets of Sobolev capacity zero. In particular $D(E)\ge0$.
\end{lemma}

We will also use the following estimates from~\cite[Lemmas 4.2, 5.2, and Theorem 4.3]{BDPV}.
\begin{lemma}\label{lem:BDPV}
Let $d\ge2$. Then every bounded open set $U$ with $|U|=\om$ satisfies
\begin{equation}\label{eq:alphashape}
 |U\sym B_1(b_U)|^2\le C_d\alpha(U)\,,
\end{equation}
and every open set $E$ of positive finite measure satisfies
\begin{equation}\label{eq:preliminaryA}
 \mathcal A(E)^4\le C_dD(E)\,.
\end{equation}
For each $M\ge2$, there exist $\eta(d,M),C(d,M)>0$ such that, if $U,V\subset B_M(z)$ are open sets of volume $\om$, then
\begin{equation}\label{eq:alphaLip}
 |\alpha(U)-\alpha(V)|\le C(d,M)|U\sym V|\,,
\end{equation}
and if $\alpha(U)\le\eta(d,M)$, then
\begin{equation}\label{eq:alpha}
 \alpha(U)\le C(d,M)(\mathcal E(U)-\mathcal E(B))\,.
\end{equation}
\end{lemma}

Our next lemma comes from Allen--Kriventsov--Neumayer~\cite[Section 5.1, Section 5.2, and Lemma 3.1, (3.3)]{AKN}.
\begin{lemma}\label{lem:AKN}
Let $d\ge2$. For each bounded open set $U\subset\R^d$ with $|U|=\om$, there is a point $x_U$, the truncated barycenter of $U$, such that
\begin{equation}\label{eq:resolvent}
 \|u_U^f-u_{B_1(x_U)}^f\|_{L^2(\R^d)}^2\le C_d(\mathcal E(U)-\mathcal E(B))
\end{equation}
for every $f\in C_c^\infty(\R^d)$ with $0\le f\le1$. Here $u_G^f\in H^1_0(G)$ solves $-\Delta u_G^f=f$ weakly and is extended by zero. In particular,
\begin{equation}\label{eq:AKN}
 \|w_U-w_{B_1(x_U)}\|_{L^2(\R^d)}^2\le C_d(\mathcal E(U)-\mathcal E(B))\,.
\end{equation}
Moreover, there exists $\eta_d>0$ such that, if $U,V$ are bounded open sets of volume $\om$ and $|U\sym B_1(z)|+|V\sym B_1(z)|\le\eta_d\om$ for some $z\in \R^d$, then
\begin{equation}\label{eq:center}
 |x_U-x_V|\le C_d|U\sym V|\,.
\end{equation}
For every $z\in\R^d$, $x_{B_1(z)}=z$.
\end{lemma}

The following lemma shows that a set with small torsion deficit contains a bounded subset obtained by removing only a small amount of volume. This subset is close to a ball centered at its barycenter, and its torsion function satisfies both an interior gradient estimate and a full-space $L^2$ estimate relative to that ball.

\begin{samepage}
\begin{lemma}\label{lem:prepare}
There are $\eta_d,C_d,M_d>0$ with the following property. If $E$ is open, $|E|=\om$, and $0\le D(E)\le\eta_d$, then there exists a bounded open subset $F\subset E$. Let $r^d=|F|/\om$ and $c=|F|^{-1}\int_Fx\dd$. Then
\begin{align}
 |E\setminus F|+|1-r|&\le C_dD(E),\qquad F\subset B_{M_d}(c),\label{eq:cutmass}\\
 |F\sym B_r(c)|&\le C_d\sqrt{D(E)},\label{eq:cutshape}\\
 \int_F\left|\nabla w_F+\frac{x-c}{d}\right|^2\dd
 +\|w_F-w_{B_r(c)}\|_{L^2(\R^d)}^2&\le C_dD(E)\,.\label{eq:cutenergy}
\end{align}
The actual subset $F$ is retained in these statements; it is not replaced by a dilation when used alongside another set.
\end{lemma}
\end{samepage}
\begin{proof}
If $D(E)=0$, the equality case of the Saint--Venant inequality in~\eqref{eq:SV} gives a ball up to a set of Sobolev capacity zero, with the same torsion function. Openness excludes points of $E$ outside that ball, so $E$ is bounded. With $F=E$ and $r=1$, the ball formula~\eqref{eq:balltorsion} gives~\eqref{eq:cutmass}--\eqref{eq:cutenergy}. Assume $D=D(E)>0$.

Translate $E$ so that a minimizing comparison unit ball is centered at the origin, and put $b_k=|E\setminus B_k|/\om$. By~\eqref{eq:preliminaryA}, decreasing $\eta_d$ ensures $b_1\le1/2$. From the proof of~\cite[Lemma 5.3, (5.8), (5.11), (5.12)]{BDPV}, we obtain, uniformly for integers $k\ge1$,
\begin{equation}\label{eq:cutrecurrence}
 b_1\le C_dD^{1/4},\quad b_{k+2}\le C_d(D+b_k^{1+1/d}),\quad
 \mathcal E(E\cap B_{k+2})\le\mathcal E(E)+C_db_k^{1+1/d}\,.
\end{equation}

A number of iterations depending only on the dimension reduces the mass to $O_d(D)$. Specifically, let $\gamma=1+1/d$. By induction in~\eqref{eq:cutrecurrence} and using $D\le1$, we can obtain
\begin{equation*}
 b_{1+2j}\le C_{d,j}(D+D^{\gamma^j/4})\,.
\end{equation*}
Fix $j=j(d)$ with $\gamma^j\ge4$ and let $k=1+2j$. Then $F=E\cap B_{k+2}$ has $|E\setminus F|\le C_dD$, and
\begin{equation*}
 \mathcal E(F)\le\mathcal E(E)+C_dD^{1+1/d}\,.
\end{equation*}
Normalize this bounded subset for the stability estimates: put $r^d=|F|/\om$ and $U=r^{-1}F$. Then $r=1+O_d(D)$, $|U|=\om$, $\diam U\le M_d$, and $D(U)\le C_dD$ by the scaling of $\mathcal E$.

We check the smallness hypothesis in~\eqref{eq:alpha}. By~\eqref{eq:preliminaryA}, there is a unit ball $B_1(z)$ with $|U\sym B_1(z)|\le C_dD(U)^{1/4}$. Once $D$ is small, this ball intersects $U$, so both lie in a fixed dimensional ball after translation. By~\eqref{eq:alphaLip} and $\alpha(B_1(z))=0$, we have $\alpha(U)\le C_dD(U)^{1/4}$. Thus~\eqref{eq:alpha} applies, and
\begin{equation}\label{eq:preparedalpha}
 \alpha(U)\le C_dD\,.
\end{equation}
By~\eqref{eq:alphashape} and~\eqref{eq:preparedalpha}, the ordinary barycenter $b=b_U$ satisfies
\begin{equation}\label{add}
 |U\sym B_1(b)|\le C_d\sqrt D\,.
\end{equation}

To estimate the interior gradient defect, we use $w_U\in H^1_0(U)$ to integrate against $x-b$. This gives $\int_U(x-b)\cdot\nabla w_U=-dT(U)$. Expanding the square and using $T(B)=d^{-2}\int_B|x|^2\dd$ yields
\begin{align}\label{eq:momentidentity}
 \int_U\left|\nabla w_U+\frac{x-b}{d}\right|^2\dd
 &=T(B)-T(U)
 +\frac1{d^2}\left(\int_U|x-b|^2\dd-\int_{B_1(b)}|x-b|^2\dd\right).
\end{align}
Equal volumes express the last difference as $\int_{U\sym B_1(b)}\big||x-b|^2-1\big|\dd$. Since $b$ lies in the convex hull of $U$, the diameter bound controls $|x-b|+1$ on $U\sym B_1(b)$. Therefore, the last difference is bounded by $C_d\alpha(U)$. Using~\eqref{eq:preparedalpha} and $D(U)\le C_dD$, both terms on the right of~\eqref{eq:momentidentity} are bounded by $C_dD$.

Next we decrease $\eta_d$ if necessary so that the preceding symmetric-difference bound \eqref{add} meets the smallness requirement in~\eqref{eq:center}. Applying that estimate to $U$ and $B_1(b)$ gives $|x_U-b|\le C_d\sqrt D$. Thanks to the translation inequality
\begin{equation*}
 \|v(\cdot-h)-v\|_2\le |h|\|\nabla v\|_2\qquad(v\in H^1(\R^d)),
\end{equation*}
which follows by integrating directional derivatives, we can deduce from \eqref{eq:AKN} that $\|w_U-w_{B_1(b)}\|_2^2\le C_dD$.

Under the dilation back to $F$, the torsion is multiplied by $r^2$, the gradient by $r$, and volume by $r^d$. As $r=1+O_d(D)$ and the image of $b$ is exactly $c$, we conclude the proof of~\eqref{eq:cutmass}--\eqref{eq:cutenergy}.
\end{proof}

The next proposition bounds the overlap of the comparison balls for two disjoint sets.

\begin{samepage}
\begin{proposition}\label{prop:contact}
Let $F_1,F_2\subset\R^d$ be disjoint bounded open sets. Suppose $w_i=w_{F_i}$, $c_i\in\R^d$, and $s\ge0$ satisfy
\begin{equation}\label{eq:contacthyp}
 \sum_{i=1}^2\left[
 \int_{F_i}\left|\nabla w_i+\frac{x-c_i}{d}\right|^2\dd
 +\|w_i-w_{B_1(c_i)}\|_2^2\right]\le s\,.
\end{equation}
If $0<\tau=2-|c_1-c_2|\le1/4$, then $s\ge c_d\Phi_d(\tau)$.
\end{proposition}
\end{samepage}
\begin{proof}
Put $R=1-\tau/8$, $D_i=B_R(c_i)$, $p_i(x)=(1-|x-c_i|^2)/(2d)$, and $t=\tau/(32d)$. Since $\inf_{D_i}p_i=(16\tau-\tau^2)/(128d)>3t$, define
\begin{equation}\label{eq:saturation}
 \varphi_i=\min\{(p_i-w_i)_+,t\}\quad\hbox{on }D_i\,.
\end{equation}
The zero extensions $w_i$ belong to $H^1(\R^d)$, and the weak chain rule gives
\begin{equation*}
 \nabla\varphi_i=\mathbf1_{\{0<p_i-w_i<t\}}(\nabla p_i-\nabla w_i)
 \quad\hbox{almost everywhere.}
\end{equation*}
On the set $\{0<p_i-w_i<t\}$, $w_i>p_i-t>2t>0$, so this set lies in $F_i$ up to a null set. Moreover, $0\le\varphi_i\le|p_i-w_i|$ and $p_i=w_{B_1(c_i)}$ on $D_i$. Hence
\begin{equation}\label{eq:satH1}
 \begin{aligned}
 \int_{D_i}|\nabla\varphi_i|^2\dd
 &\le\int_{F_i}\left|\nabla w_i+\frac{x-c_i}{d}\right|^2\dd,\\
 \int_{D_i}\varphi_i^2\dd&\le\|w_i-w_{B_1(c_i)}\|_2^2\,.
 \end{aligned}
\end{equation}
From~\eqref{eq:satH1} and the Sobolev extension theorem \cite[Theorem 9.7]{Brezis}, we can obtain extensions $\widetilde\varphi_i\in H^1(\R^d)$ with $\sum_i\|\widetilde\varphi_i\|_{H^1(\R^d)}^2\le C_ds$. Thus $v=\widetilde\varphi_1+\widetilde\varphi_2$ satisfies $\|v\|_{H^1(\R^d)}^2\le C_ds$. Disjointness of $F_1$ and $F_2$ gives $\min(w_1,w_2)=0$ almost everywhere. Since $p_i>3t$, at least one $\varphi_i$ equals $t$ at almost every point of $D_1\cap D_2$, and therefore
\begin{equation}\label{eq:obstacle}
 v\ge t\quad\hbox{almost everywhere on }D_1\cap D_2\,.
\end{equation}

After a rigid motion, we may take $c_{1,2}=\mp(1-\tau/2)e_1$. Set $a=\sqrt\tau/4$, $\delta=\tau/8$, and
\begin{equation*}
 K=(-\delta,\delta)\times B_a^{d-1},\qquad
 Q=(-a,a)\times B_a^{d-1}\,.
\end{equation*}
Since $\delta/a=\sqrt\tau/2\le1/4$, we have $K\subset Q$. Moreover $K\subset D_1\cap D_2$, because
\begin{equation*}
 (1-\tau/8)^2-\bigl[(1-3\tau/8)^2+\tau/16\bigr]
 =7\tau/16-\tau^2/8>0\,.
\end{equation*}
Write $v_A=|A|^{-1}\int_Av\dd$ for the average over a measurable set $A$ of positive finite measure. The function $h(z)=|B_a^{d-1}|^{-1}\int_{B_a^{d-1}}v(z,y)\,dy$ belongs to $H^1(-a,a)$. The fundamental theorem of calculus and the Cauchy--Schwarz inequality give
\begin{equation*}
 |v_K-v_Q|\le\int_{-a}^a|h'(z)|\,dz
 \le\frac{(2a)^{1/2}}{|B_a^{d-1}|^{1/2}}\|\partial_1v\|_{L^2(Q)}
 \le C_da^{(2-d)/2}\|\nabla v\|_2\,.
\end{equation*}
This estimate is uniform in $0<\delta\le a$. Since $Q\subset B_{2a}$ and $|Q|\asymp_d a^d$, the Poincar\'e inequality on $B_{2a}$ also gives $|v_Q-v_{B_{2a}}|\le C_da^{(2-d)/2}\|\nabla v\|_2$. Consequently, by~\eqref{eq:obstacle},
\begin{equation}\label{eq:overlap_average}
 t\le v_K\le |v_{B_{2a}}|+C_da^{(2-d)/2}\|\nabla v\|_2\,.
\end{equation}

For $0<r<1$, we define the vector field
\begin{equation*}
 X_r(x)=\frac{x}{d\om}
 \begin{cases}
 r^{-d}-1,&|x|\le r,\\
 |x|^{-d}-1,&r<|x|<1,\\
 0,&|x|\ge1.
 \end{cases}
\end{equation*}
It is straightforward to compute that
\begin{equation*}
 \operatorname{div}X_r=\frac{\mathbf1_{B_r}}{|B_r|}-\frac{\mathbf1_{B_1}}{|B_1|},
 \qquad
 \|X_r\|_2^2\le C_d\left(r^{2-d}+\int_r^1\rho^{1-d}\,d\rho\right)
 \le C_d\begin{cases}\log(e/r),&d=2,\\r^{2-d},&d\ge3,\end{cases}
\end{equation*}
and $v_{B_r}-v_{B_1}=-\int_{\R^d}X_r\cdot\nabla v\dd$. Since $|v_{B_1}|\le C_d\|v\|_2$, it follows that
\begin{equation}\label{eq:ball_average}
 |v_{B_r}|\le C_d\|v\|_{H^1}
 \begin{cases}\sqrt{\log(e/r)},&d=2,\\r^{(2-d)/2},&d\ge3.\end{cases}
\end{equation}
Take $r=2a$. Combining~\eqref{eq:overlap_average} and~\eqref{eq:ball_average} with $\|v\|_{H^1}^2\le C_ds$ gives $s\ge c_dt^2a^{d-2}=c_d\tau^{(d+2)/2}$ for $d\ge3$. For $d=2$, it gives $s\ge c_dt^2/\log(e/(2a))\ge c_d\tau^2/\log(e/\tau)$, since $2a=\sqrt\tau/2\ge\tau$ for $\tau\le1/4$.
\end{proof}

In the planar construction in the proof of Theorem \ref{thm:sharp}, we will use the following Rayleigh estimate for Lipschitz deformations of the unit ball $B$.

\begin{samepage}
\begin{lemma}\label{lem:rayleigh}
Let $V\in W^{1,\infty}(B;\R^d)$ and $F=I+V$. If $\|DV\|_\infty\le\eta_d$ for a sufficiently small dimensional $\eta_d>0$, then $F$ is bi-Lipschitz and
\begin{equation}\label{eq:rayleigh}
 0\le |F(B)|^{2/d}\lambda_1(F(B))-\om^{2/d}\lam
 \le C_d\int_B|DV|^2\dd\,.
\end{equation}
\end{lemma}
\end{samepage}
\begin{proof}
Since \(B\) is convex, the assumption \(\|DV\|_{L^\infty(B)}\le\eta_d\) implies
\[
(1-\eta_d)|x-y|\le |F(x)-F(y)|
\le(1+\eta_d)|x-y|
\]
for all \(x,y\in B\). Thus \(F\) is bi-Lipschitz. Let $u$ be the positive $L^2$-normalized first eigenfunction on $B$. It is smooth and radial, with $\partial_\nu u=-k$ on $\partial B$ for some $k>0$. Put $A=DV$, $Q=\int_B|A|^2\dd$, and $b=\int_B\operatorname{tr}A\dd$. For $\eta_d$ sufficiently small, the pointwise matrix expansions are
\begin{equation*}
 \det(I+A)=1+\operatorname{tr}A+O_d(|A|^2),
 \qquad (I+A)^{-T}=I-A^T+O_d(|A|^2)\,.
\end{equation*}
Use $u\circ F^{-1}\in H^1_0(F(B))$ as a test function. The change of variables $y=F(x)$ gives the Jacobian $\det(I+A)$ and the transformed gradient $(I+A)^{-T}\nabla u$. Hence
\begin{align*}
 |F(B)|&=\om+b+O_d(Q),\\
 \int_{F(B)}(u\circ F^{-1})^2&=1+m+O_d(Q),\\
 \int_{F(B)}|\nabla(u\circ F^{-1})|^2&=\lam+n+O_d(Q),
\end{align*}
where
\begin{equation*}
 m=\int_Bu^2\operatorname{tr}A\dd,
 \qquad n=\int_B\bigl(|\nabla u|^2\operatorname{tr}A
 -2\nabla u\cdot A\nabla u\bigr)\dd\,.
\end{equation*}
By the Cauchy--Schwarz inequality, we have $|b|+|m|+|n|\le C_d\sqrt Q$. Therefore, we get
\begin{align*}
 |F(B)|^{2/d}
 &=\om^{2/d}\left(1+\frac{2b}{d\om}\right)+O_d(Q),\\
 \frac{\int_{F(B)}|\nabla(u\circ F^{-1})|^2}
 {\int_{F(B)}(u\circ F^{-1})^2}
 &=\lam+n-\lam m+O_d(Q)\,,
\end{align*}
and their product is
\begin{equation*}
 \om^{2/d}\lam+\om^{2/d}\left(n-\lam m+\frac{2\lam}{d\om}b\right)+O_d(Q)\,.
\end{equation*}

Set the matrix field $S=(|\nabla u|^2-\lam u^2)I-2\nabla u\otimes\nabla u$. Direct differentiation gives
\begin{equation*}
 \operatorname{div}S=-2(\Delta u+\lam u)\nabla u=0\,.
\end{equation*}
On $\partial B$, we have $u=0$ and $\nabla u=-k\nu$, so $S\nu=-k^2\nu$. Integration by parts thus yields
\begin{equation*}
 n-\lam m=\int_B S:DV\dd
 =\int_{\partial B}(S\nu)\cdot V\,d\sigma
 =-k^2\int_{\partial B}V\cdot\nu\,d\sigma=-k^2b\,.
\end{equation*}
By Rellich's identity~\cite[p.~635, (1)]{Rellich1940}, we also have
\begin{equation*}
 2\lam=\int_{\partial B}(x\cdot\nu)(\partial_\nu u)^2\,d\sigma
 =d\om k^2\,.
\end{equation*}
Consequently, $n-\lam m+2\lam b/(d\om)=0$. Therefore, the Rayleigh quotient of the test function $u\circ F^{-1}\in H^1_0(F(B))$ gives the upper bound in~\eqref{eq:rayleigh}, and the Faber--Krahn inequality~\eqref{eq:classicalFK} gives the lower bound.
\end{proof}
The following lemma uses the nearly tangent two-ball configuration of Brasco and Pratelli~\cite[Example 3.4]{BP}. It shows that the Fraenkel $2$-asymmetry remains comparable to $\varepsilon$ when the balls are replaced by disjoint subsets whose total missing volume is at most a sufficiently small multiple of $\varepsilon$.
\begin{lemma}\label{lem:matching}
Let $d\ge2$. Then there exist $\eps_d,\delta_d,c_d,C_d>0$ with the following property. Let $0<\eps<\eps_d$, set $c_\pm=\pm(1-\eps)e_1$, and let $E_\pm\subset B_1(c_\pm)$ be disjoint open sets. If
\begin{equation*}
 m_\eps=\sum_\pm|B_1(c_\pm)\setminus E_\pm|\le\delta_d\eps,
\end{equation*}
then $c_d\eps\le\A(E_+\cup E_-)\le C_d\eps$.
\end{lemma}

\begin{proof}
Put $m=m_\eps$, $E=E_+\cup E_-$, and
\[
r=\left(\frac{|E|}{2\om}\right)^{1/d}
=\left(1-\frac{m}{2\om}\right)^{1/d}
=1+O_d(m).
\]
The two tangent balls $B_r(\pm re_1)$ are admissible
in~\eqref{eq:defA}. By the triangle inequality, we have
\[
\begin{aligned}
&|E\sym(B_r(re_1)\cup B_r(-re_1))|\\
&\quad\le m+\sum_\pm
|B_1(c_\pm)\sym B_r(\pm re_1)|\\
&\quad\le C_d(\eps+m)\le C_d\eps.
\end{aligned}
\]
This gives the upper bound.

For the lower bound, let
$Q=B_r(q_+)\cup B_r(q_-)$ be any admissible pair.
If one of these balls does not meet $E$, then
$|Q\sym E|\ge\om r^d\ge\om/2$.
Otherwise, $Q\cup E\subset B_4$.
Since $|q_+-q_-|\ge2r$, we have
$|q_+|^2+|q_-|^2\ge2r^2$, and hence
\[
\int_Q|x|^2\dd
\ge2\om\left(1+\frac{d}{d+2}\right)r^{d+2}.
\]
On the other hand, $E\subset B_1(c_+)\cup B_1(c_-)$
gives
\[
\int_E|x|^2\dd
\le2\om\left((1-\eps)^2+\frac{d}{d+2}\right).
\]
Using
\[
r^{d+2}\ge
1-\frac{d+2}{d}\frac{m}{2\om},
\]
we obtain
\[
\begin{aligned}
16|Q\sym E|
&\ge\int_Q|x|^2\dd-\int_E|x|^2\dd\\
&\ge4\om\eps-2\om\eps^2
-\frac{2(d+1)}{d}m\\
&\ge c_d\eps,
\end{aligned}
\]
provided that $\eps_d$ and $\delta_d$ are
sufficiently small. Taking the infimum over
$Q$ and dividing by $|E|$ gives the lower
bound.
\end{proof}

\section{Proofs of the main results}\label{sec:proofs}
In this section we present the proofs of Theorems \ref{thm:main} and \ref{thm:sharp}.
\subsection{Proof of Theorem~\ref{thm:main}}\label{sec:global}
\begin{proof}
Note that both quantities in the statement are scale invariant. We assume $|\Omega|=2\om$ and put $s=\lambda_2(\Omega)/\lam-1\ge0$, by~\eqref{eq:classicalHKS}. Since $\A\le2$ and $\Phi_d$ is bounded on $[0,2]$, it is enough to consider $s$ below a fixed positive dimensional threshold.

The spectral splitting of~\cite[Lemma 7.49]{BDP}, valid for finite-measure open sets, gives disjoint open subsets $E_1,E_2\subset\Omega$ with $\lambda_1(E_i)\le\lambda_2(\Omega)$; see also~\cite[Lemma 3.1]{BP}. Write $|E_i|=\om R_i^d$. The Faber--Krahn inequality~\eqref{eq:classicalFK} gives $R_i^d\ge(1+s)^{-d/2}$. Since $R_1^d+R_2^d\le2$, we also have $R_i^d\le2-(1+s)^{-d/2}$. Consequently,
\begin{equation}\label{eq:balance}
 |R_i-1|\le C_ds,\qquad |\Omega\setminus(E_1\cup E_2)|\le C_ds\,,
\end{equation}
and $0\le R_i^2\lambda_1(E_i)/\lam-1\le C_ds$. Applying \eqref{eq:SV} and \eqref{eq:KJ} to $R_i^{-1}E_i$ gives $0\le T(B)-T(R_i^{-1}E_i)\le C_ds$, and hence $D(R_i^{-1}E_i)\le C_ds$.

Applying a scaling form of Lemma~\ref{lem:prepare}, we can obtain disjoint bounded sets $F_i\subset E_i$ with barycenters $c_i$. Set $r_i^d=|F_i|/\om$. Then from~\eqref{eq:cutmass}--\eqref{eq:cutenergy} and~\eqref{eq:balance} we can derive that
\begin{align}
 |\Omega\setminus(F_1\cup F_2)|+\sum_i|r_i-1|&\le C_ds,\label{eq:physicalmass}\\
 \sum_i|F_i\sym B_{r_i}(c_i)|&\le C_d\sqrt s,\label{eq:physicalshape}\\
 \sum_i\left[\int_{F_i}\left|\nabla w_{F_i}+\frac{x-c_i}{d}\right|^2\dd
 +\|w_{F_i}-w_{B_{r_i}(c_i)}\|_2^2\right]&\le C_ds\,.\label{eq:physicalenergy}
\end{align}
From~\eqref{eq:physicalmass}--\eqref{eq:physicalshape}, we can further obtain
\begin{equation}\label{eq:unitenergy}
 \sum_i|F_i\sym B_1(c_i)|
 \le\sum_i\bigl(|F_i\sym B_{r_i}(c_i)|+\om|r_i^d-1|\bigr)
 \le C_d\sqrt s\,.
\end{equation}
By the formula~\eqref{eq:balltorsion}, the two torsion functions with radii $r_i$ and $1$ differ pointwise by at most $C_d|r_i-1|$ and have support in $B_2(c_i)$. Thus we have
\begin{equation}\label{eq:radiusL2}
 \sum_i\|w_{B_{r_i}(c_i)}-w_{B_1(c_i)}\|_2^2
 \le C_d\sum_i|r_i-1|^2\le C_ds^2\,.
\end{equation}
Combining~\eqref{eq:physicalenergy} and~\eqref{eq:radiusL2}, we obtain~\eqref{eq:contacthyp} with right-hand side $C_ds$.

Set $\tau=(2-|c_1-c_2|)_+$. Disjointness of $F_1,F_2$ and~\eqref{eq:unitenergy} give
\begin{equation*}
 |B_1(c_1)\cap B_1(c_2)|\le\sum_i|B_1(c_i)\setminus F_i|\le C_d\sqrt s\,.
\end{equation*}
Since $s$ is small, we have $\tau<1/4$. From Proposition~\ref{prop:contact} we get $\Phi_d(\tau)\le C_ds$.

Let $n=(c_2-c_1)/|c_2-c_1|$, $q_1=c_1-\tau n/2$, and $q_2=c_2+\tau n/2$. Then $|q_1-q_2|\ge2$ and $\sum_i|q_i-c_i|=\tau$. The balls $B_1(q_i)$ are admissible in~\eqref{eq:defA}.Since a translation by $h$ changes a unit ball by at most $C_d|h|$ in measure, we can combine~\eqref{eq:physicalmass} and~\eqref{eq:unitenergy} to deduce
\begin{align}\label{eq:Afinal}
 2\om\A(\Omega)
 &\le |\Omega\setminus(F_1\cup F_2)|+\sum_i|F_i\sym B_1(c_i)|+C_d\sum_i|q_i-c_i|\\
 &\le C_d(\sqrt s+\tau)\,.\notag
\end{align}
The formulas in~\eqref{eq:Phi} show that $\Phi_d$ is increasing near zero and, for each fixed $K\ge1$, satisfies $\Phi_d(Ku)\le C_{d,K}\Phi_d(u)$ for sufficiently small $u$. Since $u+v\le2\max\{u,v\}$, this also gives $\Phi_d(u+v)\le C_d(\Phi_d(u)+\Phi_d(v))$ near zero. Moreover $\Phi_d(\sqrt s)\le C_ds$. It follows from~\eqref{eq:Afinal} and Proposition~\ref{prop:contact} that
\begin{equation*}
 \Phi_d(\A(\Omega))\le C_d\bigl(\Phi_d(\sqrt s)+\Phi_d(\tau)\bigr)\le C_ds\,,
\end{equation*}
\enlargethispage{\baselineskip}
which concludes the proof.
\end{proof}

\subsection{Proof of Theorem~\ref{thm:sharp}}\label{sec:optimality}
\begin{proof}
Assume first that $d\ge3$, and put $q=(d+2)/2$, $c_\eps=(1-\eps)e_1$, and $P_\eps=B_1(c_\eps)\cap\{x_1>0\}$. By van den Berg~\cite[Theorem 1(ii), (13)--(14)]{vdB}, we have
\begin{equation}\label{eq:capdeficit}
 0\le |P_\eps|^{2/d}\lambda_1(P_\eps)-\om^{2/d}\lam
 \le C_d\eps^q.
\end{equation}

Since $P_\eps$ is a bounded convex domain, it admits
a smooth convex inner exhaustion; see
~\cite[Theorem A]{Blocki1997}.
By domain monotonicity and the variational
characterization of the first Dirichlet eigenvalue,
we may choose, for each sufficiently small $\eps>0$,
a smooth convex domain $E_\eps\Subset P_\eps$ such that
\begin{equation}\label{eq:inneraccuracy}
 |P_\eps\setminus E_\eps|\le\eps^2,\qquad
 0\le\lambda_1(E_\eps)-\lambda_1(P_\eps)\le\eps^q.
\end{equation}
Since $|E_\eps|\le|P_\eps|$, we have:
\begin{align*}
 |E_\eps|^{2/d}\lambda_1(E_\eps)-\om^{2/d}\lam
 &\le |P_\eps|^{2/d}\lambda_1(P_\eps)-\om^{2/d}\lam\\
 &\quad+|P_\eps|^{2/d}\bigl(\lambda_1(E_\eps)-\lambda_1(P_\eps)\bigr)
 \le C_d\eps^q.
\end{align*}
We can also compute that
\begin{equation*}
 a_\eps=|B_1(c_\eps)\setminus P_\eps|
 =\omega_{d-1}\int_0^\eps(2t-t^2)^{(d-1)/2}\,dt
 \asymp_d\eps^{(d+1)/2}.
\end{equation*}
Set $\Omega_\eps=E_\eps\cup(-E_\eps)$. Its two components are smooth, bounded, connected, and congruent. Their closures have positive distance because $E_\eps\Subset\{x_1>0\}$. The total missing volume in Lemma~\ref{lem:matching} is
\begin{equation*}
 m_\eps=2a_\eps+2|P_\eps\setminus E_\eps|
 \le C_d\eps^{(d+1)/2}+2\eps^2=o(\eps).
\end{equation*}
For sufficiently small $\eps$, $m_\eps\le\delta_d\eps$, so Lemma~\ref{lem:matching} gives $\A(\Omega_\eps)\asymp_d\eps$. The direct-sum spectrum, counted with multiplicity, gives $\lambda_2(\Omega_\eps)=\lambda_1(E_\eps)$, and therefore
\begin{equation}\label{eq:highdimdeficit}
 \HKS(\Omega_\eps)
 =2^{2/d}\bigl(|E_\eps|^{2/d}\lambda_1(E_\eps)-\om^{2/d}\lam\bigr)
 \le C_d\eps^q.
\end{equation}
The corresponding lower bound follows from Theorem~\ref{thm:main}, already proved in Subsection~\ref{sec:global}.

Now let $d=2$. Fix $c=(1-\eps)e_1$, $z=-\eps e_1$,
$R=1/8$, $a=\eps^{1/4}$, and $B^+=B_1(c)$.
Choose smooth cutoffs $\chi_{\rm in},\chi_{\rm out}$
valued in $[0,1]$, with $\chi_{\rm in}(t)=0$ for
$t\le1$ and $\chi_{\rm in}(t)=1$ for $t\ge2$,
and $\chi_{\rm out}(t)=1$ for $t\le R/2$
and zero for $t\ge R$. Set
\begin{equation*}
 v_a(t)=\frac{\chi_{\rm in}(t/a)\chi_{\rm out}(t)}{t},
 \qquad L_a=\int_0^R v_a(t)\,dt,
 \qquad g(r)=\frac{2\eps}{L_a}\int_r^R v_a(t)\,dt
\end{equation*}
for $r\le R$, and $g=0$ beyond $R$. Then
\begin{equation}\label{eq:planarlog}
 L_a=\log(R/a)+O(1)
 =\tfrac14|\log\eps|+O(1).
\end{equation}
The function $g$ is smooth, nonnegative, and
nonincreasing, and equals $2\eps$ on $[0,a]$.
Define
\begin{equation}\label{eq:push}
 F(x)=x+g(|x-z|)e_1,\qquad
 E^+=F(B^+),\qquad E^-=-E^+,\qquad
 \Omega_\eps=E^+\cup E^-.
\end{equation}
Since $g$ is constant near zero, the map $F$ is smooth near $z$. Moreover,
$\|g'\|_\infty\le C\eps/(aL_a)\to0$,
so $F$ is a smooth bi-Lipschitz map.
Both image components are smooth, bounded,
connected, and congruent.

For $x\in\overline{B^+}$ with $x_1\le\eps$,
we have
\[
 |x-z|^2\le 2(x_1+\eps)
 \le4\eps<a^2
\]
for sufficiently small $\eps$. Since
$x_1\ge-\eps$ and $g=2\eps$ on $[0,a]$,
we obtain
\[
 F_1(x)=x_1+2\eps\ge\eps.
\]
For $x_1>\eps$, the same inequality follows
from $g\ge0$. Consequently,
\[
 \overline{E^+}\subset\{x_1\ge\eps\},
 \qquad
 \overline{E^-}\subset\{x_1\le-\eps\}.
\]
Thus the two components have disjoint
closures with distance at least $2\eps$.

By radial integration
and~\eqref{eq:planarlog}, we can compute that
\[
 \int_{B^+}|D(F-I)|^2\dd
 \le\frac{C\eps^2}{L_a^2}
       \int_a^R\frac{dr}{r}
 \le\frac{C\eps^2}{L_a}
 \le\frac{C\eps^2}{|\log\eps|}.
\]
Applying Lemma~\ref{lem:rayleigh} after translating
$B^+$ to $B$, we obtain
\begin{equation}\label{eq:planarcomponent}
 0\le |E^+|\lambda_1(E^+)-\om\lam
 \le\frac{C\eps^2}{|\log\eps|}.
\end{equation}

Note that along each line parallel to $e_1$, the first
coordinate of $F$ is strictly increasing.
For $y\in(-1,1)$, the left endpoint of the
ball section is
\[
 L(y)=1-\eps-\sqrt{1-y^2}
\]
and moves right by $g(r_L(y))$, where
$r_L(y)=|(L(y),y)-z|\ge|y|$.
The right endpoint is unchanged, since its
distance from $z$ is at least one.
Consequently $E^+\subset B^+$, and
\begin{align}\label{eq:missing}
 |B^+\setminus E^+|
 &=\int_{-1}^1 g(r_L(y))\,dy
 \le2\int_0^R g(r)\,dr\notag\\
 &=\frac{4\eps}{L_a}\int_0^R t v_a(t)\,dt
 \le\frac{C\eps}{L_a}=o(\eps).
\end{align}

By~\eqref{eq:missing}, Lemma~\ref{lem:matching}
applies for sufficiently small $\eps$ and gives
$\A(\Omega_\eps)\asymp\eps$.
Since the two components are congruent and
disjoint, the direct-sum spectrum gives
$\lambda_2(\Omega_\eps)=\lambda_1(E^+)$.
From~\eqref{eq:planarcomponent}, we obtain
\[
 \HKS(\Omega_\eps)
 =2\bigl(|E^+|\lambda_1(E^+)-\om\lam\bigr)
 \le\frac{C\eps^2}{|\log\eps|}.
\]
The matching lower bound follows from
Theorem~\ref{thm:main}, since
$\Phi_2(\A(\Omega_\eps))\asymp\Phi_2(\eps)$.
\end{proof}

\begingroup
\makeatletter
\let\@tocwrite\@gobbletwo
\makeatother
\section*{Declarations}

\noindent\textbf{Conflict of interest.}
The authors declare that they have no conflicts of interest.

\medskip

\noindent\textbf{Data availability.}
Data sharing is not applicable to this article, as no datasets were generated or analyzed during the current study.

\medskip

\noindent\textbf{Generative AI disclosure.}
The initial research ideas and overall direction of this work originated with the authors. Generative AI tools were used to assist with research discussions, checking mathematical arguments and references, and preparing and editing the manuscript, including its LaTeX source. The authors reviewed the AI-assisted material and take full responsibility for the mathematical results and the final manuscript.
\endgroup

\bibliographystyle{amsplain}
\bibliography{HKS_StageX_Final_References}

\end{document}